\documentclass[12pt]{article}
\usepackage{amssymb, amsmath, amsthm, epsf, epsfig, color}
\usepackage{graphicx}
\usepackage[usenames,dvipsnames]{xcolor}
\usepackage{soul}
\usepackage{alltt}
\usepackage{calc,ifthen,pbox}

\newlength{\originalbase}
\begin{document}
\title{Cops vs.\ Gambler}

\author{Natasha Komarov\thanks{Department of Mathematical Sciences, Carnegie Mellon University,
Pittsburgh PA 15213, USA; natasha.komarov@gmail.com.}
\, and Peter Winkler\thanks{Department of Mathematics, Dartmouth College,
Hanover NH 03755-3551, USA; peter.winkler@dartmouth.edu. Research supported
by NSF grant DMS-1162172, and in part by a Simons Professorship at the Mathematical
Sciences Research Institute, Berkeley CA.}}

\maketitle

\newtheorem{theorem}{Theorem}[section]
\newtheorem{claim}{Claim}[theorem]
\newtheorem{prop}[theorem]{Proposition}
\newtheorem{remark}[theorem]{Remark}
\newtheorem{lemma}[theorem]{Lemma}
\newtheorem{corollary}[theorem]{Corollary}
\newtheorem{guess}[theorem]{Conjecture}
\newtheorem{conjecture}[theorem]{Conjecture}

\begin{abstract}
We consider a variation of cop vs.\ robber on graph in which the robber is not restricted by the graph edges;
instead, he picks a time-independent probability distribution on $V(G)$ and moves according to this fixed distribution.
The cop moves from vertex to adjacent vertex with the goal of minimizing expected capture time. Players move simultaneously.
We show that when the gambler's distribution is known, the expected capture time (with best play) on any connected $n$-vertex
graph is exactly $n$.  We also give bounds on the (generally greater) expected capture time when the gambler's distribution is unknown to the cop.
\end{abstract}

\section{Introduction}
\label{intro section}
The game of cops and robbers on graphs was introduced independently by Nowakowski and Winkler \cite{NW} and Quilliot \cite{Q},
and has generated a great deal of study in the three decades since; see, e.g., \cite{1,BGHK,3,4}.
In the original formulation a cop and robber move alternately from vertex to adjacent vertex (or stay where
they are) on a connected, undirected graph $G$. The players have full information about each other's
current position at each step. The cop's goal is to minimize capture time, the robber's to maximize it.
There are graphs on which a robber playing optimally can elude the cop forever; for instance, chasing the robber
on the 4-cycle is clearly a hopeless endeavor for the cop. Graphs on which a cop can win are called ``cop-win.''
More precisely, a graph is cop-win if there is a vertex $u$ such that for every vertex $v$, the cop beginning at
$u$ can capture the robber beginning at $v$.  Cop-win graphs---also known as ``dismantlable'' graphs \cite{NW}---have
appeared in statistical physics \cite{BW, congress} as well as combinatorics and game theory.

The capture time in the original version of the game played on a cop-win graph has been analyzed and found to be at most
$n{-}4$ for all graphs with $n \geq 7$ vertices~\cite{BGHK,Gavenciak}. This game contains equitable restrictions on the
movements of the two players: the cop and robber are both constrained by the graph and can both see each other.
What happens to the capture time if the rules are asymmetrical, and/or the game is played ``at night''?
In the ``hunter and rabbit'' game~\cite{ARSSV, kakeya}, the players move without seeing each other, and the robber-turned-rabbit
is not constrained by the graph edges; that is, he is free to move to any vertex of the graph at each step.
It turns out that the rabbit has a strategy that will get him expected capture time $\Omega(n\log n)$ on the
$n$-cycle (or any graph of linear diameter).

Here we consider a pursuit game with the following rules.  The game is played on a graph $G$ (which will be
assumed throughout this work to be connected and undirected) with $V(G) = \{v_1,v_2, \dots, v_n\}$.  The cop is constrained
to the graph as above, moving from vertex to adjacent vertex (or staying put) at each step. The robber, whom we will now call a {\bf gambler},
chooses a probability distribution $p_1, p_2, \dots, p_n$ on $V(G)$ so that at each time $t \geq 0$, he is at vertex
$v_i$ with probability $p_i$. We call this probability distribution his {\bf gamble}. The players move simultaneously
and in the dark and the game continues until the gambler is captured---that is, until the players occupy the same vertex
at the same time.  The {\bf capture time} is the number of moves up to and including the capture, thus a positive integer;
the cop's objective is to minimize expected capture time while the gambler does his best to maximize it, thus we may think
of the expected capture time, with best play by both players, as the {\bf value} of the game (to the gambler).  We consider
two variations: one in which the cop knows the gamble, and one in which she does not. When the gamble is known by the cop,
we have the following rather surprising result: the value of the cop vs.\ gambler game is exactly $n$
regardless of the graph structure or the method of choosing the cop's initial location.

Pursuit games have obvious application in warfare (e.g., destroyer vs.\ submarine) and crime-fighting, but our somewhat
less adversarial cop vs.\ gambler game is perhaps more likely to appear in software design.  Imagine, for example, that
an anti-incursion program has to navigate a linked list of ports, trying to minimize the time to intercept an enemy packet
as it arrives.  If the enemies' port-choice distribution is known we get a version of the known gambler, otherwise the
unknown gambler.

\section{Cop \& gambler on a tree}
We suppose first that the graph $G$ on which the game is played is a tree with vertices $v_1,\dots,v_n$, with the cop beginning
at vertex $v_1$ (which we think of as the root).

\begin{lemma}
\label{tree}
The cop can capture the gambler in expected time at most $n$ on any tree of order $n$.
\end{lemma}

\begin{proof}
For any $i$ and $j$, we let $P_{ij}$ be the (unique) path from $v_i$ to $v_j$.  We denote by $B_i$ the branch of $G$
beginning at $v_i$, that is, $B_i := \{v_j:~v_i \in P_{1j}\}$.  Let $m_i := |B_i|$ be the number of
vertices in that branch and $c_i := \sum_{v_j \in B_i}p_j$ the sum of the probabilities assigned to that branch by the gamble.

We will now (re)-number the vertices of $G$ so that the cop's strategy will be to follow the path $v_1, v_2,\dots,v_k$ from the root
toward a leaf, possibly stopping for good at some vertex on the way.  The path is defined inductively as follows: given $v_1,\dots,v_i$,
let $v_{i+1}$ be a neighbor of $v_i$, other than $v_{i-1}$, that maximizes $c_i/m_i$.  (Informally, the cop enters a branch with
vertices of highest average probability.)  If there is no such $u$, i.e., if $v_i$ is a leaf with $i>1$, then $k=i$ and the path-labeling
is finished; the remaining vertices of $G$ are numbered arbitrarily.

Let $T_i$ be the capture time (always assuming best play) from the moment the cop moves to $v_i$; we will prove, by backward induction
on $i$, that if $v_i$ is reached, then $T_i \le m_i/c_i$.  Note that $m_1 {=} n$ and $c_1 {=} 1$, thus the claim is equivalent to the
statement of the lemma for $i=1$.

If the cop reaches the leaf $v_k$, she remains there and since capture is now a matter of waiting for success in a sequence of
i.i.d.\ Bernoulli trials with success probability $p_k$, we have $T_k = 1/p_k = m_k/c_k$, establishing the base of the induction.

Suppose the cop is at $v_i$, $i<k$. If $p_i \ge c_i/m_i$, the cop stays at $v_i$ and captures in expected time $1/p_i \le m_i/c_i$ as claimed.
Otherwise she moves on to $v_{i+1}$ giving
$$
T_i \le 1 + (1 - p_i)T_{i+1} \le 1 + (1 - p_i)\frac{m_{i+1}}{c_{i+1}}
$$
by the induction assumption.  

Since the average probability of vertices in $B_i \setminus \{v_i\}$ is $\frac{c_i - p_i}{m_i - 1}$, and $v_{i+1}$ was chosen to maximize
average probability in $B_{i+1}$, we know that $\frac{c_{i+1}}{m_{i+1}} \ge \frac{c_i - p_i}{m_i - 1}$.  Hence,
$$
T_i \le 1 + (1 - p_i)\frac{m_i - 1}{c_i - p_i} = 1 + (m_i - 1)\frac{1 - p_i}{c_i - p_i}~.
$$
Noting that $\frac{1 - p_i}{c_i - p_i}$ decreases as $p_i$ decreases, and recalling that $p_i < c_i/m_i$, we deduce that
$$
T_i \le 1 + (m_i - 1)\frac{1 - c_i/m_i}{c_i - c_i/m_i} = \frac{m_i}{c_i}
$$
and the proof is complete.
\end{proof}

\section{Cop \& gambler on a general graph}
In this section, we complete our capture time calculations for the cop vs.\ gambler game by showing that the gambler can achieve
expected capture time at least $n$ on any connected, $n$-vertex graph, and consequently that the value of the cop vs.\ gambler game is exactly $n$.

\begin{lemma}
\label{atmostn}
The cop and gambler game played against a uniform gamble has expected capture time exactly $n$.
\end{lemma}

\begin{proof}
With $p_i = 1/n$ for every $i$, the game is a sequence of i.i.d.\ Bernoulli trials with success probability $1/n$, hence expected capture time $n$,
irrespective of the cop's strategy.
\end{proof}

\begin{theorem}
\label{general}
The value of the cop vs.\ gambler game on any connected $n$-vertex graph is $n$.
\end{theorem}

\begin{proof}
Let $G$ be any connected graph of size $n$ and let $H$ be a spanning subtree of $G$. By Lemma~\ref{tree}, the cop can capture the gambler
in expected time at most $n$ on $H$, and consequently on $G$. By Lemma~\ref{atmostn}, we know that the expected capture time is also at least $n$.
\end{proof}

We have not said how the cop's initial position is chosen, but we can now deduce that it does not matter.

\begin{lemma}
The value of the cop vs.\ gambler game is $n$ in all three of the following cases:
\begin{enumerate}
	\item[(a)] The cop's initial position is chosen for her ahead of time.
	\item[(b)] The cop chooses her own initial position, but before she knows the gambler's strategy.
	\item[(c)] The cop chooses her initial position after the gambler makes his strategy known.
\end{enumerate}
\end{lemma}

\begin{proof}
	Note that the first situation is the worst for the cop (as she has no control over her starting point) and the third situation is
the best (as she has as much information as is possible before the start of the game). Therefore it suffices to show that the capture time
is at most $n$ in situation (a) and at least $n$ in situation (c). 
	\begin{enumerate}
	\item[(a)] In this situation, the cop can still get expected capture time $n$ on a tree by Lemma~\ref{tree} and consequently on any graph. 
	\item[(c)] In this situation, the expected capture time is still at least $n$ by Lemma~\ref{atmostn}.
	\end{enumerate}
\end{proof}

\section{Cop \& unknown gambler}
To get an expected capture time of at most $n$ when playing against a gambler with a public gamble, our cop relied on knowing the gambler's
strategy. What if this strategy were not known to the cop? We will call the adversary in this variant of the game the ``unknown gambler''
and will proceed to consider the expected capture time in this case.

Let us note first that an additional restriction on the unknown gambler, namely that he may choose only a delta distribution, is equivalent
to a previously-studied problem \cite{mobile&immobile, thesis} which we call ``cop vs.\ sitter."  Here the gambler simply picks a vertex
and hides there.  Expected capture time on an $n$-vertex tree for this game is again $n$ (the way we count moves in this paper); it is not hard to show
that given the cop's initial position, there is a unique meta-probability distribution (concentrated on leaves) for the gambler's choice of
hiding place that achieves the game value.  If the graph is not a tree, expected capture time is strictly less than $n$, minimized at $n/2 + 1$
on the cycle $C_n$ (or any edge-supergraph of it); the cop can achieve this by circling in a random direction.

Without the delta-restriction, the unknown gambler seems strictly stronger than the known gambler, but much weaker than the ``rabbit'' who can
jump to any vertex at any time, not following a fixed probability distribution.  We will show that intuition is correct in both cases.

\begin{lemma}
\label{ukg<rabbit}
The unknown gambler is strictly weaker than the rabbit on $C_n$.
\end{lemma}

\begin{proof}
We know from \cite{ARSSV,kakeya} that the rabbit can force expected capture time $\Theta(n \log n)$ on $C_n$, and it is easy to see that the cop
can achieve linear expected capture time against the unknown gambler on the same graph.  The cop's strategy (which is in fact optimal---see
\cite{thesis} for a proof) is simply to run around and around the cycle, clockwise or counterclockwise according to a fair coin-flip.  Each
circuit of the cop takes $n$ steps and has capture probability
$$
1 - \prod_{i=1}^n (1-p_i) \ge 1 - \left(1-\frac{1}{n}\right)^n \ge 1 - \frac{1}{e}~.
$$
Thus an average of only $\frac{e}{e-1}$ circuits are needed, and the mean length of the circuit on which capture occurs is at most $n/2 + 1$.
Hence expected capture time is at most $1 + \big( \frac{e}{e-1} - \frac12 \big) \sim 1.082n$.
\end{proof}

It turns out (see \cite{thesis}) that the unknown gambler can achieve this expected capture time (asymptotically), by choosing uniformly at random an
interval of about $\sqrt{n}$ vertices on the circle and making his gamble uniform on that interval.  Rather than reproduce this (fairly difficult)
proof in order to show that the unknown gambler is strictly stronger than the unknown gambler, we consider the easier case of the star $K_{1,n-1}$.
\begin{lemma}
\label{ukg>g}
The unknown gambler is strictly stronger than the known gambler on $K_{1,n-1}$.
\end{lemma} 

\begin{proof}
We know of course that expected capture time for cop vs.\ known gambler on $K_{1,n-1}$, or anything else, is $n$; suppose the cop can achieve
this against the unknown gambler as well.  One of the unknown gambler's strategies is that of the random sitter, in which he selects a leaf $u$ uniformly
at random and chooses the probability distribution concentrated at $u$.  To counter this the cop must search the leaves systematically, that is,
without repetition; for example, by visiting leaves in random order on even moves. Since each new leaf requires two steps, this just gets the expected
capture time $n$ that the players are entitled to.  But such a strategy fails against the gamble that assigns the uniform probability distribution to
leaves, because the result is i.i.d.\ Bernoulli trials with success probability $\frac{1}{n-1}$ every {\em other} turn, giving expected capture time $2n-2$.
(A good compromise on the cop's part is choosing a random permutation of the leaves and then visiting each leaf {\em twice}, getting expected capture
time about $3n/2$ versus both the random sitter and the uniform gambler.)
\end{proof}

From lemmas \ref{ukg<rabbit} and \ref{ukg>g} we immediately have:

\begin{theorem}
The unknown gambler is strictly stronger than the known gambler and strictly weaker than the rabbit.
\end{theorem}

We conclude with a linear upper bound for expected capture time against the unknown gambler, suggested by the star analysis above.

\begin{lemma}
The cop vs.\ unknown gambler game has expected capture time less than $2n$, on any connected $n$-vertex graph.
\end{lemma}

\begin{proof}
It suffices to prove the result when $G$ is a tree.  Let the cop perform a {\em depth first search} on $G$; that is, a tour of
$G$ in which once a branch is entered, it is not exited until all of its vertices have been visited.  A coin is flipped to determine whether
the cop will perform the search forward or backward; but in either case, each time she enters a leaf, she stays there for an extra turn.
The search is repeated until capture.

Since there are at most $n{-}1$ leaves and each edge is traversed twice, one search takes time at most $3n{-}2$.  During that time every
vertex is visited at least twice, hence the probability of capture in a given round is at least
$$
1 - \prod_{i=1}^n (1-p_i)^2 \ge 1 - \left(1-\frac{1}{n}\right)^{2n} \ge 1 - \frac{1}{e^2}~.
$$

Thus an average of $1/(1-e^{-2}) \sim 1.157$ rounds will suffice for capture, and the average time to capture in the {\em successful}
round is at most $3n/2$.  We conclude that overall the capture time is bounded by $3n/(1-e^{-2}) - 3n/2 \sim 1.97n$.
\end{proof}

Our guess?  The right bound is $3n/2$, the star being the worst case.

\bibliographystyle{amsplain}
\bibliography{copsandrobbersbib}

\providecommand{\bysame}{\leavevmode\hbox to3em{\hrulefill}\thinspace}
\providecommand{\MR}{\relax\ifhmode\unskip\space\fi MR }
\providecommand{\MRhref}[2]{%
  \href{http://www.ams.org/mathscinet-getitem?mr=#1}{#2}
}
\providecommand{\href}[2]{#2}
\begin{thebibliography}{10}

\bibitem{ARSSV}
Micah Adler, Harald R{\"a}cke, Naveen Sivadasan, Christian Sohler, and Berthold
  V{\"o}cking, \emph{Randomized pursuit-evasion in graphs}, Combin. Probab.
  Comput. \textbf{11} (2003), no.~3, 225--244.

\bibitem{kakeya}
Y.~Babichenko, Y.~Peres, R.~Peretz, P.~Sousi, and P.~Winkler, \emph{Hunter,
  {C}auchy rabbit, and optimal {K}akeya sets}, Trans.\ Amer.\ Math.\ Soc., to
  appear, 2013.

\bibitem{1}
A.~Berarducci and B.~Intrigila, \emph{On the cop number of a graph}, Adv. Appl.
  Math. \textbf{14} (1993), 389--403.

\bibitem{BGHK}
A.~Bonato, P.~Golovach, G.~Hahn, and J.~Kratochvil, \emph{The capture time of a
  graph}, Discrete Math. \textbf{309} (2009), 5588--5595.

\bibitem{BW}
G.R. Brightwell and P.~Winkler, \emph{Maximum hitting time for random walks on
  graphs}, J. Rand. Struct. and Alg. \textbf{1} (1990), no.~3, 263--276.

\bibitem{congress}
\bysame, \emph{Hard constraints and the {B}ethe lattice: adventures at the
  interface of combinatorics and statistical physics}, Proc. Int'l. Congress of
  Mathematicians Vol. III, Li Tatsien, ed. (2002), 605--624.

\bibitem{mobile&immobile}
Shmuel Gal, \emph{Search games with mobile and immobile hider}, SIAM J. Control
  \& Opt. \textbf{17} (1979), no.~1, 99--122.

\bibitem{Gavenciak}
T.~Gaven\u{c}iak, \emph{Games on graphs}, Master's thesis, Charles University,
  Prague, 2007.

\bibitem{3}
G.~Hahn, F.~Laviolette, N.~Sauer, and R.E. Woodrow, \emph{On cop-win graphs},
  Discrete Math. \textbf{258} (2002), 27--41.

\bibitem{4}
G.~Hahn and G.~MacGillivray, \emph{A note on k-cop, l robber games on graphs},
  Discrete Math. \textbf{306} (2006), 2492--2497.

\bibitem{thesis}
Natasha Komarov, \emph{Capture time in variants of cops and robbers games},
  Ph.D. thesis, Dartmouth College, 2013.

\bibitem{NW}
R.~Nowakowski and P.~Winkler, \emph{Vertex to vertex pursuit in a graph},
  Discrete Math. \textbf{43} (1983), 235--239.

\bibitem{Q}
A.~Quilliot, \emph{Homomorphismes, points fixes, r\'{e}tractations et jeux de
  poursuite dans les graphes, les ensembles ordonn\'{e}s et les espaces
  m\'{e}triques}, Ph.D. thesis, Universit\'{e} de Paris VI., 1983.

\end{thebibliography}

\end{document}